\documentclass[a4paper]{amsart}
\usepackage[english]{babel}
\usepackage{inputenc}
\usepackage[T1]{fontenc}
\usepackage{amsmath}
\usepackage{amsthm}
\usepackage{amsfonts}
\usepackage{amssymb}
\usepackage{amsaddr}
\usepackage{enumerate}
\usepackage{enumitem}
\usepackage{mathrsfs}
\usepackage{graphicx}
\usepackage{hyperref}
\usepackage{tikz}
\usepackage[normalem]{ulem}
\newcommand{\LL}{\mathcal{L}}
\renewcommand{\H}{\mathrm{H}}

\newcommand{\ZZ}{\mathbb{Z}}
\newcommand{\RR}{\mathbb{R}}
\newcommand{\QQ}{\mathbb{Q}}
\newcommand{\FF}{\mathbb{F}}

\newcommand{\cH}{\mathcal{H}}
\newcommand{\cP}{\mathcal{P}}

\DeclareMathOperator{\D}{DA}
\DeclareMathOperator{\Sal}{\overline{\textrm{Sal}}}
\DeclareMathOperator{\im}{im}

\DeclareMathOperator{\FC}{FC}
\DeclareMathOperator{\FP}{FP}
\DeclareMathOperator{\F}{F}
\DeclareMathOperator{\lk}{lk}

\newtheorem{teo}{Theorem}[section]
\newtheorem{dfn}[teo]{Definition}
\newtheorem{prop}[teo]{Proposition}
\newtheorem{lem}[teo]{Lemma}
\newtheorem{corol}[teo]{Corollary}

\theoremstyle{definition}
\newtheorem{exam}[teo]{Example}
\theoremstyle{remark}
\newtheorem{obs}[teo]{Remark}

\title{On the Sigma-invariants of even Artin groups of $\FC$-type}


\author[R.~Blasco, J.I.~Cogolludo, C.~Mart{\'\i}nez]{Rub\'en Blasco-Garc{\'\i}a, Jos\'e Ignacio Cogolludo-Agust{\'\i}n, and
Conchita Mart{\'\i}nez-P\'erez$^*$}
\address{
\small{
$^*$Corresponding author.\\
Departamento de Matem\'aticas, IUMA, Facultad de Ciencias, Universidad de Zaragoza, \\
c/ Pedro Cerbuna 12, E-50009 Zaragoza SPAIN
}}
\email{rubenb@unizar.es,jicogo@unizar.es,conmar@unizar.es}

\begin{document}

\thanks{\noindent
The authors are partially supported by Departamento de Ciencia, Universidad y Sociedad del 
Conocimiento del Gobierno de Arag{\'o}n (grant code: E22\_20R: ``{\'A}lgebra y Geometr{\'i}a''),
the second author is partially supported by MCIN/AEI/10.13039/501100011033 
(grant code: PID2020-114750GB-C31) and the first and third authors are partially supported by 
the Spanish Government PGC2018-101179-B-100} 
%20J06, 20F36 

\begin{abstract} 
In this paper we study Sigma-invariants of even Artin groups of $\FC$-type, extending some known results for 
right-angled Artin groups. In particular, we define a condition that we call the strong $n$-link 
condition for a graph $\Gamma$ and prove that it gives a sufficient condition for a character $\chi:A_\Gamma\to\ZZ$ 
to satisfy $[\chi]\in\Sigma^n(A_\Gamma,\ZZ)$. This implies that the kernel $A^\chi_\Gamma=\ker \chi$ 
is of type $\FP_n$. We prove the homotopical version of this result as well
and discuss partial results on the converse. We also provide a general formula for the free part of $H_n(A^\chi_\Gamma;\FF)$ as an $\FF[t^{\pm 1}]$-module 
with the natural action induced by $\chi$. This gives a characterization of when $H_n(A^\chi_\Gamma;\FF)$ is a 
finite dimensional vector space over~$\FF$.
\end{abstract}


\maketitle

\keywords

\section{Introduction}
The Sigma-invariants of a group $G$ are certain sets $\Sigma^n(G,\ZZ)$, $\Sigma^n(G)$ of 
equivalence classes of characters $\chi:G\to\RR$ that provide information about the cohomological 
--\,in the case of $\Sigma^n(G,\ZZ)$\,-- and homotopical --\,for $\Sigma^n(G)$\,--  finiteness conditions 
of subgroups lying over the commutator of $G$. The first version of these invariants was defined by Bieri 
and Strebel in~\cite{Bieri-Strebel-Valuations} and the theory was later developed by 
Bieri-Neumann-Strebel~\cite{Bieri-Neumann-Strebel-Geometric}, Bieri-Renz~\cite{Bieri-Renz-Valuations}, 
and Renz~\cite{Renz-Geometrische}.
Usually, it is extremely difficult to compute these invariants explicitly but there are some remarkable 
cases in which a full computation is available. 

One of those cases occurs when $G$ is a right-angled Artin group (RAAG for short). These groups are
defined from a given finite graph $\Gamma$ which will be assumed here to be simple, i.e., without loops 
or multiple edges between vertices. Associated to $\Gamma$ one can describe the RAAG, denoted by $A_\Gamma$, 
as the group generated by the vertices of $\Gamma$ with relators of the form $[v,w]=1$ for any edge $\{v,w\}$ 
of $\Gamma$. This is a remarkable family of groups that range between finitely generated free 
abelian groups (corresponding to complete graphs) and finitely generated free groups (associated with graphs 
with no edges). Many properties of RAAGs can be determined in terms of the combinatorial properties of 
the graph. This is precisely the case for their Sigma-invariants, which were computed by 
Meier-Meinert-VanWyk in~\cite{MMVW}. To describe their computation we will need to introduce some terminology.

We recall the concept of link in our context as follows. Fix a simple finite graph $\Gamma$ as before and 
denote by $V_\Gamma$ (resp. $E_\Gamma$) its set of vertices (resp. edges). 
If $\Gamma_1\subseteq\Gamma$ is a subgraph and $v\in V_\Gamma$, then the link 
$\lk_{\Gamma_1}(v)$ of $v$ in $\Gamma_1$ is defined as the full subgraph induced by 
$V_{\Gamma_1}(v):=\{w\in V_{\Gamma_1}\mid \{v,w\}\in E_\Gamma\}$.

We extend this definition for subsets $\Delta\subseteq\Gamma$ by setting 
$$\lk_{\Gamma_1}(\Delta)=\cap_{v\in\Delta}\lk_{\Gamma_1}(v).$$
By convention we allow $\Delta$ to be empty, then
$\lk_{\Gamma_1}(\Delta)=\Gamma_1.$

We also recall the concept of the \emph{flag complex} associated with $\Gamma$. This the simplicial complex, 
denoted as $\hat\Gamma$, resulting after attaching a ($k-1$)-simplex to each \emph{$k$-clique}, i.e., to each 
complete subgraph of $k$ vertices. We use the same notation for arbitrary graphs. 
Note that, if $\Delta\subseteq\Gamma$ is a clique and $\Gamma_1\subseteq \Gamma$ a subgraph, 
then $\hat{\lk}_{\Gamma_1}(\Delta)$ is the intersection with $\hat{\Gamma}_1$ of the ordinary simplicial 
link of the cell $\sigma$ associated to $\Delta$, i.e., the subcomplex of $\hat\Gamma_1$ consisting of those simplices 
$\tau$ such that $\tau\cup\sigma$ is also a simplex of $\hat\Gamma_1$. 

Now, let $\chi:A_\Gamma\to\RR$ be a character and $n\geq 0$ an integer. Consider the full subgraph $\LL^\chi_0$ 
induced by the vertices $v$ of $\Gamma$ with $\chi(v)\neq 0$. Following Meier-Meinert-VanWyk~\cite{MMVW}, we call 
$\LL^\chi_0$ the {\sl living} subgraph of $\Gamma$ and say that vertices not in $\LL^\chi_0$ are {\sl dead}. 
Dead vertices are also called {\sl resonant} in~\cite{Blasco-conmar-ji-Homology}. 
We will say that the character $\chi$ satisfies the \emph{$n$-link condition} if for any clique $\Delta\subseteq\Gamma\setminus\LL^\chi_0$, 
$$\hat{\lk}_{\LL^\chi_0}(\Delta)\text{ is }(n-1-|\Delta|)\textrm{-acyclic}.$$ 
Then Meier-Meinert-VanWyk proved (see Subsection~\ref{subsec:SigmaRAAGs}) that $\chi\in\Sigma^n(G,\ZZ)$ 
if and only if $\chi$ satisfies the $n$-link condition. In fact, they also proved  the homotopical version of 
this result that characterizes $\Sigma^n(G)$ in terms of a {\sl homotopical $n$-link condition} (with ``being 
$(n-1-|\Delta|)$-connected'' instead of ``being $(n-1-|\Delta|)$-acyclic'').

Here, we want to extend this result for another remarkable family: even Artin groups of $\FC$-type. 
Given a finite simple graph $\Gamma$ as above, one can consider an even labeling on the edges, that is, for 
any edge $e=\{u,v\}$, its label $\ell(e)$ is an even number. Any such \emph{even graph} $\Gamma$ defines an 
\emph{even Artin group} $A_\Gamma$ generated by the vertices of $\Gamma$ and whose relators have the form $(uv)^k=(vu)^k$, 
where $\ell(e)=2k$. These special Artin groups where first considered in detail in~\cite{Blasco} and~\cite{Blasco-PF}. 
Note that any subgraph $X\subset \Gamma$ of an even graph $\Gamma$ generates an even Artin group $A_X$. 
In addition, an even Artin group is said to have $\FC$-type if $A_X$ is of finite type for each clique $X\subset \Gamma$: 
this means that the \emph{standard parabolic Coxeter group}  $W_X$, i.e., the  quotient of $A_X$ by the normal subgroup 
generated by $\langle u^2; u\in V_X\rangle$ is finite.

For a character $\chi$ on an even Artin group we consider a generalization of the living subgraph as follows (see~\cite{M2001}).
Denote $m_v=\chi(v)$, $v\in V_\Gamma$ and $m_e=m_v+m_w$, $e=\{v,w\}\in E_\Gamma$. 
We say that an edge is {\sl dead} if  $e$ has  label $\ell(e)>2$ and $m_e=0$. 
We will consider the subgraph $\LL^\chi$ obtained from $\Gamma$ after removing all dead vertices and the interior of all dead
edges. Note that if all the edges have label precisely 2, i.e., for a RAAG, then $\LL^\chi_0=\LL^\chi$.

To state our first main result, we also need to introduce the {\sl clique poset}, that is the poset of subgroups 
of $A_\Gamma$ which are generated by cliques of $\Gamma$:
$$\cP=\{A_\Delta\mid\Delta\subseteq\Gamma\text{ clique}\}.$$
Note that the poset structure of $\cP$ is the poset structure of the poset of cliques of $\Gamma$. 
Also, we allow $\Delta$ to be empty, in that case $A_\emptyset=\{1\}$. 
So the geometric realization of the clique poset is the cone of 
the barycentric subdivision of the flag complex where the vertex of the cone corresponds to the empty clique.

A special role is played by the subset $\mathcal{B}^\chi\subset \cP$ of those subgroups $A_\Delta$ where 
$\Delta\subseteq\Gamma$ is a clique such that for each vertex $v$ in $\Delta$ either $v$ is dead or $v\in e$ for $e$ 
a dead edge in $\Delta$. Observe that $1=A_\emptyset\in\mathcal{B}^\chi$. We will see that this is equivalent to asking 
that the center of $A_\Delta$ lies in the kernel of $\chi$, that is, $\chi(Z(A_\Delta))=0$.

\begin{dfn} \label{nlink}
Let $\mathcal{B}^\chi\subset \cP$ be as above. 
Assume that for any $A_\Delta\in\mathcal{B}^\chi$ with $|\Delta|\leq n$ the link $\hat{\lk}_{\LL^\chi}(\Delta)$ 
is $(n-1-|\Delta|)$-acyclic. Then we say that $\chi$ satisfies \emph{the strong $n$-link condition}. 

We also define a \emph{homotopical strong $n$-link condition} in a similar way just changing $(n-1-|\Delta|)$-acyclic 
by $(n-1-|\Delta|)$-connected.
\end{dfn}

Note that the homotopical strong $n$-link condition implies the strong $n$-link condition.

\begin{teo}\label{teo:mainsigma} 
Let  $G=A_\Gamma$ be an even Artin group of $\FC$-type, and $0\neq\chi:G\to\RR$ a character such that the strong 
$n$-link condition holds for $\chi$. Then $[\chi]\in\Sigma^n(G,\ZZ)$.
\end{teo}

In the case $n=1$, it is known for several types of Artin groups (see Theorem~\ref{teo:knownArtin}) 
that $[\chi]\in\Sigma^n(G,\ZZ)$ if and only if $\LL^\chi$ is  {\sl connected} and {\sl dominant}. This is equivalent to saying that $\chi$ satisfies the strong $n$-link condition (see Subsection~\ref{subsec:SigmaRAAGs}).

We do not know whether the converse of Theorem~\ref{teo:mainsigma} is true in general, but in Section~\ref{sec:free} we prove a partial converse. 
To do that, we use some of the techniques of \cite{Blasco-conmar-ji-Homology} to perform computations on the homology 
groups $\H_n(A_\Gamma^\chi,\FF)$ where $\FF$ is a field, $\chi:A_\Gamma\to\ZZ$ is assumed to be discrete and $A_\Gamma^\chi=\ker\chi$. 
More precisely, we show that these homology groups are finite dimensional as $\FF$-vector spaces if and only if certain $p$-local version of the 
strong $n$-link condition holds. Recall that a consequence of the well-known properties of the Sigma-invariants (see Section~\ref{sec:Sigma}) is that 
if for a discrete character $\chi$ we have $[\chi]\in\Sigma^n(G,\ZZ)$, then $A_\Gamma^\chi$ is of type $\FP_n$ and therefore the homology groups 
with coefficients over any field must be finite dimensional. As a by-product, an explicit computation of independent interest is provided 
in Theorem~\ref{teo:free} for the free part of $\H_n(A_\Gamma^\chi,\FF)$ when seen, via $\chi$, as a module over the principal ideal 
domain~$\FF[t^{\pm 1}]$. 

Moreover, section~\ref{sec:homotopic} is devoted to stating and proving a partial homotopic analogue of
Theorem~\ref{teo:mainsigma} in Theorem~\ref{teo:mainsigmahomotopic}.

\section{Sigma-invariants}\label{sec:Sigma}

Let $G$ be a finitely generated group. In this section we will consider arbitrary non-trivial characters $\chi:G\to\RR$. 
We say that two characters $\chi_1,\chi_2$ are equivalent if one is a positive scalar multiple of the other, i.e., 
if $\chi_1=t\chi_2$ for some $t>0$. We denote by $[\chi]$ the equivalence class of the character $\chi$ and by $S(G)$ 
the set of equivalence classes of characters. 
Note that if $G/G'$ has finite torsion and free rank $r$ then $S(G)$ can be 
identified with the sphere $S^{r-1}$. 
%Assume now that $G$ is of type $\FP_\infty$ (or just $\FP_n$ if we want to define the invariants up to degree $n$ only). 
The homological $\Sigma$-invariants of $G$ are certain subsets
$$\Sigma^\infty(G,\ZZ)\subseteq\dots\subseteq\Sigma^n(G,\ZZ)\subseteq\dots\subseteq\Sigma^2(G,\ZZ)
\subseteq\Sigma^1(G,\ZZ)\subseteq \Sigma^0(G,\ZZ)=S(G)$$
which are very useful to understand the cohomological finiteness properties of subgroups of $G$ containing the commutator $G'$. 

For a formal definition, consider $\chi:G\to\RR$ a character and $G_\chi$ the monoid 
$G_\chi=\{g\in G\mid\chi(g)\geq 0\}$. Then
$$\Sigma^n(G,\ZZ)=\{[\chi]\in S(G)\mid \ZZ G_\chi\text{ is of type }\FP_n\}.$$

There is also a homotopical version 
$$\Sigma^\infty(G)\subseteq\dots\subseteq\Sigma^n(G)\subseteq\dots\subseteq
\Sigma^2(G)\subseteq\Sigma^1(G)\subseteq \Sigma^0(G)=S(G).$$
We can sketch the definition as follows (see ~\cite{MMVW}).
Let $G$ be a group of type $\F_n$. We can choose a $CW$-model $X$ for the classifying space for $G$ 
with a single 0-cell and finite $n$-skeleton. 
Let $Y$ be the universal cover of $X$. Then we may identify $G$ with a subset of $Y$ and given a 
character $\chi:G\to\RR$, we can extend $\chi$ to a map $\chi:Y\to\RR$ that we denote in the same way. 
To do that, map the vertex labeled by, say, $g$ to $\chi(g)$ and extend linearly to the rest of~$Y$.

For $a\in\RR$ denote by $Y_{\chi}^{[a,+\infty)}$ the maximal subcomplex in $Y\cap\chi^{-1}([a,+\infty))$.
Assuming $a\leq 0$, the inclusion $Y_{\chi}^{[0,+\infty)}\subseteq Y_{\chi}^{[a,+\infty)}$ induces a map
%$$\H_i(Y_{\chi}^{[0,+\infty)})\to\H_i(Y_{\chi}^{[a,+\infty)}).$$
%Then we say that $[\chi]\in\Sigma^n(G,\ZZ)$ if there is some $a$ such that the map above is trivial for all $i<n$. 
%There is also an induced map between the homotopy groups
$$\pi_i(Y_{\chi}^{[0,+\infty)})\to\pi_i(Y_{\chi}^{[a,+\infty)})$$
and we say that $[\chi]\in\Sigma^n(G)$ if there is some $a$ such that this map is trivial for all $i<n$.
The reader can find more details about $\Sigma^n(G,\ZZ)$, $\Sigma^n(G)$ in~\cite{MMVW}. We recall now only two 
well-known properties: both $\Sigma^n(G,\ZZ)$, $\Sigma^n(G)$ are open subsets of $S(G)$ that determine the cohomological 
and homotopical finiteness conditions of subgroups containing the commutator thanks to the following fundamental Theorem:

\begin{teo}\label{teo:charsigma} 
Let $G$ be a group of type $\FP_n$ and  $G'\leq N\leq G$. Then $N$ is also of type $\FP_n$ if and only if
$$\{[\chi]\in S(G)\mid \chi(N)=0\}\subseteq\Sigma^n(G,\ZZ).$$
Moreover, if $G$ is of type $\F_n$, then $N$ is of type $\F_n$ if and only if
$$\{[\chi]\in S(G)\mid\chi(N)=0\}\subseteq\Sigma^n(G).$$
\end{teo}

In particular, if $\chi:G\to\ZZ$ is discrete, we have that $\ker\chi$ is of type $\FP_n$ if and only if 
$[\chi], [-\chi]\in\Sigma^n(G,\ZZ)$.

If $R$ is a commutative ring, one can also define $R$-Sigma-invariants $\Sigma^n(G,R)$ by substituting the homology 
groups in the definition above by homology groups with coefficients in $R$. Theorem~\ref{teo:charsigma} remains 
true when $\FP_n$ is substituted by $\FP_n$ over $R$. Moreover we have
$$\Sigma^n(G)\subseteq\Sigma^n(G,\ZZ)\subseteq\Sigma^n(G,R)$$
for any $G$, $R$ and $n\geq 2$ and
$$\Sigma^1(G)=\Sigma^1(G,\ZZ)=\Sigma^1(G,R).$$

We will also need the following useful result.

\begin{lem}\label{lem:center}\cite[Lemma 2.1]{M2001} 
Let $G$ be any group of type $\F_n$  and $\chi:G\to\RR$ a character with $\chi(Z(G))\neq 0$ where $Z(G)$ is the center of $G$. 
Then $[\chi]\in\Sigma^n(G)\subseteq\Sigma^n(G,\ZZ)$.  
\end{lem}

Finally, we state here the following result which was proven by Meier-Meinert-VanWyk~\cite{MMVW}. 
This Theorem will be the main tool in the proof of Theorem~\ref{teo:mainsigma} as it was one of the main tools in their description of the Sigma-invariants for RAAGs.

\begin{teo}\cite[Theorem 3.2]{MMVW}\label{teo:sigmacomplex} 
Let $G$ be a group acting by cell-permuting homeomorphisms on a $CW$-complex $X$ with finite $n$-skeleton 
modulo $G$. Let $\chi:G\to\RR$ be a character such that  for any $0\leq p\leq n$ and any $p$-cell $\sigma$ 
of $X$ the stabilizer $G_\sigma$ is not inside $\ker\chi$. Then
\begin{enumerate}[label*=$\roman*)$]
\item
If $X$ is $(n-1)$-connected and $[\chi|_{G_\sigma}]\in\Sigma^{n-p}(G_\sigma)$ for any $p$-cell $\sigma$,
$0\leq p\leq n$, then $[\chi]\in\Sigma^n(G)$.
\item
If $X$ is $(n-1)$-$R$-acyclic and $[\chi|_{G_\sigma}]\in\Sigma^{n-p}(G_\sigma,R)$ for any $p$-cell $\sigma$, 
$0\leq p\leq n$, then $[\chi]\in\Sigma^n(G,R)$.
\end{enumerate}
\end{teo}

\section{Artin groups and their Sigma-invariants}

As we have seen in the introduction, Artin groups can be defined in terms of a labeled graph.
Using the symmetry of the standard presentation of an Artin group we can show the following.

\begin{prop}\label{prop:symmetricsigma} 
Let $G=A_\Gamma$ be an Artin group. Then $-\Sigma^n(G)=\Sigma^n(G)$ and $-\Sigma^n(G,\ZZ)=\Sigma^n(G,\ZZ)$.
\end{prop}

\begin{proof} 
% Let $Y$ be the universal cover of the Salvetti complex of $A_\Gamma$. The fact that the presentation of $G$ is symmetric 
% respect to inversions, i.e. each relator 
% $$(xy)^\ell=(yx)^\ell\text{ or }(xy)^\ell x=(yx)^\ell y$$
% can also be written as
% $$(x^{-1}y^{-1})^\ell=(y^{-1}x^{-1})^\ell\text{ or }(x^{-1}y^{-1})^\ell x^{-1}=(y^{-1}x^{-1})^\ell y^{-1}$$
% implies that the map $G\to G$ induced by $g\mapsto g^{-1}$ induces a (not $G$-equivariant) homeomorphism $\varphi:Y\to Y$. 
% For any character $\chi:G\to\RR$, if we also denote $\chi:Y\to\RR$ the induced character on $Y$ we have $-\chi\varphi=\chi$ so 
% $\varphi(Y_{-\chi}^{[a,+\infty)})=Y_\chi^{[a,+\infty)}$ and this implies the result. 
Due to the symmetry of the relations in $G$, there is a well-defined map $\varphi:G\to G$ given as 
$\varphi(v):=v^{-1}$ for $v\in V_\Gamma$, which defines an automorphism of $G$. 
Since $\chi\circ\varphi=-\chi$ and both $\Sigma^n(G)$ and $\Sigma^n(G,\ZZ)$ are invariant under automorphisms of $G$, 
the result follows.
\end{proof}

So we have:

\begin{lem}\label{lem:discrete} 
Let $G=A_\Gamma$ be an Artin group and $\chi:G\to\ZZ$ a discrete character. 
Then  $\ker\chi$ is of type $\FP_n$ if and only if $[\chi]\in\Sigma^n(G,\ZZ)$. 
In particular, if there is some field $\FF$ and some $0<i\leq n$ such that $\dim_\FF\H^i(\ker\chi,\FF)$ is infinite, 
$\chi\not\in\Sigma^n(G,\ZZ)$.  
\end{lem}
\begin{proof} 
The first statement is a direct consequence of Theorem~\ref{teo:charsigma} and 
Proposition~\ref{prop:symmetricsigma}. For the second one, recall that if a group is of type $\FP_n$,
then after tensoring a finite type resolution of the trivial module by $\FF$, one obtains
a finite type resolution of projective modules over its group ring and thus it is also $\FP_n$ over~$\FF$. 
\end{proof}

\subsection{Sigma-invariants for RAAGs}\label{subsec:SigmaRAAGs}
The explicit computation of the Sigma-invariants for a particular group is usually very difficult. 
In~\cite{MVW1995}, Meier and VanWyk computed $\Sigma^1(A_\Gamma)$ for $A_\Gamma$ a RAAG:

\begin{teo}[Meier-VanWyk~\cite{MVW1995}]
\label{thm:MVW95}
Let $G=A_\Gamma$ be a RAAG and $\chi:G\to\RR$ a character. Then
$$[\chi]\in\Sigma^1(G)\text{ if and only if }\LL^\chi_0\text{ is connected and dominating in }\Gamma.$$
\end{teo}

Recall that $\LL^\chi_0$ is the subgraph obtained from $\Gamma$ by removing the vertices $v$ with $\chi(v)=0$. 
As we are assuming that $A_\Gamma$ is a RAAG, $\LL^\chi_0=\LL^\chi$ is the living subgraph defined above. 
Also, we say that a subgraph $\Delta\subseteq\Gamma$ is {\sl dominating} if for any $v\in\Gamma\setminus\Delta$ 
there is some $w\in\Delta$ linked to $v$. In other words, the condition of $\LL^\chi_0$ being dominant is 
equivalent to saying that for every $v\in\Gamma\setminus\LL^\chi_0$,
$\lk_{\LL^\chi_0}(v)\neq\emptyset.$ And therefore the Theorem can be reformulated as follows:
$[\chi]\in\Sigma^1(G)$ if and only if 
\begin{itemize}
\item[i)] $\lk_{\LL^\chi_0}(\emptyset)$ is 0-connected and,

\item[ii)] for every $v\in\Gamma\setminus\LL^\chi_0$, $\lk_{\LL^\chi_0}(v)$ is (-1)-connected.
\end{itemize}

This can be restated using the 1-link condition defined in the introduction:
$$[\chi]\in\Sigma^1(G)\text{ if and only if }\chi\text{ satisfies the homotopical 1-link condition}.$$

Later on, in~\cite{MMVW} Meier-Meinert-VanWyk, extending Theorem~\ref{thm:MVW95}, 
were able to give a full description of the higher Sigma-invariants of a RAAG in terms of the $n$-link condition.

\begin{teo} 
Let $G=A_\Gamma$ be a RAAG and $\chi:G\to\RR$ a character. Then $[\chi]\in\Sigma^n(G,\ZZ)$ 
if and only if the $n$-link condition holds for $\chi$.  
\end{teo}

\subsection{Some partial results for Artin groups}\label{subsec:SigmaArtin}
Not much is known about $\Sigma$-invariants of general Artin groups. 
Let $\chi:A_\Gamma\to\RR$, $A_\Gamma$ an Artin group and $Z(S)$ the center of $S\subset A_\Gamma$. 
We highlight the following partial result.

\begin{teo}[{Meier-Meinert-VanWyk~\cite[Theorem B]{M2001}}]
Assume $A_\Gamma$ is of $\FC$-type. If $\hat\Gamma$ is $(n-1)$-acyclic and $\chi(Z(A_\Delta))\neq 0$ 
for any $\emptyset\neq\Delta\subseteq\Gamma$ clique, then $[\chi]\in\Sigma^n(G,\ZZ)$.
\end{teo}

We will see below that the hypothesis $\chi(Z(A_\Delta))\neq 0$ 
for any $\emptyset\neq\Delta\subseteq\Gamma$ clique means that $\mathcal{B}^\chi$ 
consists of the trivial subgroup only. Therefore this result is a particular case of our main Theorem~\ref{teo:mainsigma}.

A full characterization is available in few cases only. 

\begin{teo}[Meier-Meinert-VanWyk~\cite{M2001}, Almeida~\cite{Almeida18}, Almeida-Kochloukova~\cite{AK15,AK15b}, Kochloukova~\cite{Kochloukova20}]
\label{teo:knownArtin}
Assume that one of the following conditions holds:
\begin{itemize}
%\item $\Gamma$ connected and $\pi_1(\Gamma)=1$,
\item $\Gamma$ is a connected tree,

\item $\Gamma$ is connected and $\pi_1(\Gamma)$ is free of rank at most 2, 

\item $\Gamma$ is even and whenever there is a closed reduced path in $\Gamma$ with all labels bigger than 2, 
then the length of such path is always odd.
\end{itemize}
Then $[\chi]\in\Sigma^1(A_\Gamma)\iff\LL^\chi$ is connected and dominating.

\end{teo}

Moreover, the class of Artin groups that satisfy the hypothesis in Theorem~\ref{teo:knownArtin} 
is known to be closed under graph products and, as a consequence, every $\FC$-type Artin group also does
(\cite{A21}). Other concrete examples of Artin groups satisfying this hypothesis can be found in \cite{AK15b} 
and \cite{A17}.

Note that by a similar observation as above, this result can be stated as follows: For
$\Gamma$ connected and with $\pi_1(\Gamma)=1$ or free of 
rank at most 2, then 
$$[\chi]\in\Sigma^1(A_\Gamma)\iff\chi\text{ satisfies the strong 1-link condition}.$$

Observe also that here we are not assuming that $A_\Gamma$ is even.

\subsection{An easier particular case: direct products of Artin dihedral groups}\label{subsec:productdihedral}
 
It will be important below to understand the Sigma-invariants of the finite type Artin subgroups $A_\Delta$ of a given even Artin 
group of $\FC$-type $A_\Gamma$. In general, finite type Artin groups are direct products of finite type irreducible Artin groups and 
the only possible irreducible finite type Artin groups are those of dihedral type, which are the Artin groups associated to a 
single edge. In the even case the edge is labeled by an even integer, say $2\ell$ and the associated group is 
%$$D_{2\ell}=\langle x,y\mid (xy)^{\tilde{\ell}}=(yx)^{\tilde{\ell}}\rangle.$$
$$\D_{2\ell}=\langle x,y\mid (xy)^{\ell}=(yx)^{\ell}\rangle.$$

The (homotopical) Sigma-invariants for irreducible Artin groups of finite type have been described in \cite[Section 2]{M2001}.  
In the particular case of a dihedral Artin group we have the following result.

\begin{lem}\cite[Pg 76]{M2001}\label{lem:dihedral}  Let $G=\D_\ell$ be a dihedral Artin group and $n\geq 1$. For any commutative ring $R$
\begin{itemize}
\item[i)] If $\ell$ is odd, then 
$S(G)$ is a 0-sphere and $S(G)=\Sigma^n(G)=\Sigma^n(G,R)$ for any $n$.

\item[ii)] If $\ell=2\tilde{\ell}$ is even $S(G)$ is a 1-sphere. Denoting by $x$, $y$ the standard generators, we have
 $\Sigma^n(G)=\Sigma^n(G,R)=S(G)\setminus\{[\chi],[-\chi]\}$ where $\chi(x)=1$, $\chi(y)=-1$.
 \end{itemize}
\end{lem}
\begin{proof} For the homotopical result, see \cite[p.~76]{M2001}. For the homological one note that $\Sigma^1(G)=\Sigma^1(G,R)$
and $$\Sigma^n(G)\subseteq\Sigma^n(G,R)\subseteq\Sigma^1(G,R).$$
\end{proof}

If $A_\Gamma$ is an even Artin group of $\FC$-type and $A_\Delta$ is a finite type subgroup with $\Delta\subseteq\Gamma$, then $\Delta$ 
must be a clique and $A_\Delta$ is a direct product of even Artin dihedral groups and possibly a factor which is free abelian of finite rank. 
In this subsection we will give a full description of the Sigma-invariants of such an $A_\Delta$. But we will consider the slightly more 
general case of a product of {\sl arbitrary} Artin dihedral groups and possibly a free abelian groups of finite rank. Assume
$$G=G_1\times\dots \times G_s$$
where each of the $G_i$'s is either $\ZZ$ or Artin dihedral. 

Using \cite[Theorem 1.4]{BG10} and the fact that according to Lemma~\ref{lem:dihedral} the $R$-Sigma-invariants for Artin 
dihedral groups for $R=\ZZ$ and $R=\QQ$ coincide, we deduce (the upper script $c$ means the complementary of the 
corresponding subset)
$$\Sigma^n(G,\ZZ)^c=\bigcup_{n_1+\ldots+n_s=n,n_i\geq 0}\Sigma^{n_1}(G_1,\ZZ)^c\star\dots\star\Sigma^{n_s}(G_s,\ZZ)^c,$$
where $\star$ is the {\sl join} product in the corresponding spheres (see \cite{BG10}). 
We have $\Sigma^m(G_i,\ZZ)^c=\emptyset$ unless 
$m\geq 1$ and $G_i$ is dihedral of even type, in that case $\Sigma^m(G_i,\ZZ)^c=\{[\chi_i],-[\chi_i]\}$ where $\chi_i$ maps the standard 
generators of $G_i$ to $1$ and $-1$ resp.  As a consequence, if we order the factors so that $G_1,\ldots,G_t$ are precisely those which are 
dihedral of even type, and denote by $v_{t+1},\ldots v_s$ the vertices not involved in dihedral type edges, we have

$$
\Sigma^m(G,\ZZ)^c=\{[\chi]\in S(G)\mid\chi(Z(G)=0\}\text{ if } m\geq t.
$$
Note that $\chi(Z(G))=0$ is equivalent to  $m_e=0\text{ for } e\in E_\Gamma \text{ if } \ell(e)>2\text{ and }\chi(v_i)=0\text{ for }t+1\leq i\leq s.$ For $m<t$,
$\Sigma^m(G,\ZZ)^c$ consists of characters that vanish in all the vertices $v_i$ for $t+1\leq i\leq s$ and in at least $t-m$ of the factors $G_i$ for $1\leq i\leq t$, having zero $m_e$-value in the rest.

\subsection{Coset posets}

In this section we prove some results on coset posets that will be used in the main proofs later on. 

\begin{dfn} 
Let $G$ be a group and $\cP$ a poset (ordered by inclusion) of subgroups of $G$.
The {\sl coset complex} $C_G(\mathcal{P})$ (or simply $C(\mathcal{P})$ if the group $G$ is clear by the context) 
is the geometric realization of the poset $G\mathcal{P}$ of cosets $gS$ where $g\in G$ and $S\in\cP$. In other words, 
it is the geometric realization of the simplicial complex whose $k$-simplices are the chains
\begin{equation}\label{eq:simplices}
g_0S_0\subset g_1S_1\subset\dots\subset g_kS_k=g_0(S_0\subset S_1\subset\dots\subset S_k),
\end{equation}
where $g_0,\ldots,g_k\in G$ and $S_0,\ldots,S_k\in\mathcal{P}$. Let $\mathcal{P}_\chi$ be the subposet of 
$\mathcal{P}$ consisting of those subgroups $S\in\cP$ such that $\chi|_S\neq 0$. 
It yields a subcomplex $C(\mathcal{P})_\chi$  of $C(\mathcal{P})$. 
We identity $C(\mathcal{P})$ with its geometric realization. 
\end{dfn}

We will consider posets of subgroups $\cP$ having a {\sl height} function $h:\cP\to\ZZ^+\cup\{0\}$ such that whenever 
$S\subsetneq T$ both sit in $\cP$, we have $h(S)<h(T)$. We also assume that there is a bound for the height of the 
elements of $\cP$. We denote that bound by $h(\cP)$.
Now, assume we have a subposet $\cH\subseteq\cP$. Then $C(\cH)$ is a subcomplex of $C(\cP)$. We want to compare the 
homology of $C(\cP)$ with the homology of $C(\cH)$. Let $$\sigma:g(S_0\subset S_1\subset\dots\subset S_k)$$
be a $k$-simplex in $C(\cP)$. Using the height function $h$ we set
$$h_{\cH^c}(\sigma)=
\begin{cases} 
-1 & \text{ if }S_i\in\cH \text{ for every }0\leq i\leq k,\\
\mathrm{max}\{h(S_j)\mid S_j\not\in\cH\} & \text{ in other case}.\\
\end{cases}$$

Note that if $\tau\subseteq\sigma$ is a face, then $h_{\cH^c}(\tau)\leq h_{\cH^c}(\sigma)$.
This implies that we can use this function to define a subcomplex
$$D^s:=\{\sigma\in C(\cP)\mid h_{\cH^c}(\sigma)\leq s\}$$
for $-1\leq s\leq h(\cP)$. 
So we have a filtration
$$C(\mathcal{P}_\chi)=D^0\subseteq D^1\subseteq\dots\subseteq D^{h(\cP)}=C(\mathcal{P}).$$
For $s\geq 0$, simplices in $D^{s}$ but not in $D^{s-1}$ are of the form 
$$\sigma:g(S_0\subset S_1\subset\dots\subset S_k)$$
such that there is some $0\leq i\leq k$ with $S_i\in\cP\setminus\cH$ of height precisely $s$ and $S_j\in\cH$ for $j>i$. 
 
Fix an $S\in\cP\setminus\cH$ of height precisely $s$ and consider the set of all simplices of the form
\begin{equation}\label{simplex}
\sigma:g(S_0\subseteq S_1\subseteq\dots \subseteq S_i\subseteq S_{i+1}\subseteq\dots\subseteq S_k)
\end{equation}
with $S_i=S$ and $S_j\in\cH$ for each $j>i$. Those simplices lie in $D^s$. The boundary $\partial\sigma$ of such a $k$-simplex $\sigma$ consists of $k-1$-simplices $\tau$ which are of the same form except of the case when 
$$\tau:g(S_0\subseteq S_1\subseteq\dots\subseteq S_{i-1}\subseteq S_{i+1}\subseteq\dots\subseteq S_k)$$
and then $\tau\in D^{s-1}$. 
Consider now the complex  $Z^S$ which is the geometric realization of $C_S(\cP_S)\setminus \{S\}$ 
for the poset $\cP_S=\{S\cap T\mid T\in\cP\}$ (note that $Z^S$ could be empty if $S$ is empty). 
And let $\mathcal{J}^S$ be the poset 
$$\mathcal{J}^S:=\{T\in\cH\mid S\subseteq T\}.$$
The join $Z^S\star|\mathcal{J}^S|$ is in a natural way an $S$-space and we may form the induced $G$-space $G/S\times(Z^S\star|\mathcal{J}^S|)$. Its chain complex is a $G$-complex that consists of induced modules of the form $\mathcal{C_\bullet}(Z^S\star|\mathcal{J}^S|)\uparrow_S^G$. We claim that the quotient complex $D^{s+1}/D^s$ can be decomposed as
\begin{equation}\label{complexS}
\bigoplus_{S\in\cP\setminus\cH,h(S)=s}\tilde{\mathcal{C}}_{\bullet +1}(G/S\times(Z^S\star|\mathcal{J}^S|)
\end{equation} 
where $\tilde{\mathcal{C}}_{\bullet +1}$ is the augmented chain complex shifted by one.  To see it, consider a simplex $\sigma$ as in~\eqref{simplex} and put $g=xy$ for $y\in S$ so that
$$\sigma:x(yS_0\subseteq\dots \subseteq yS_{i-1}\subseteq S\subseteq S_{i+1}\subseteq\dots\subseteq S_k).$$
Then $\sigma$ yields a free direct summand of the chain complex of $D^{s+1}/D^s$ at dimension $k$ 
that we map onto the summand of $\tilde{\mathcal{C}}_{\bullet +1}(G/S\times(Z^S\star|\mathcal{J}^S|)$ associated to
$x\otimes(\sigma_1\star\sigma_2)$ with 
$$\sigma_1:y(S_0\subseteq\dots\subseteq S_{i-1})$$
and 
$$\sigma_2: S_{i+1}\subseteq\dots\subseteq S_k.$$

\begin{lem}\label{lem:cosetposets} 
With the previous notation, assume that for any $S\in\cP\setminus\cH$ with $h(S)=s$ we have that 
$Z^S$ is $(s-2)$-acyclic or empty if $s=0$ and $|\mathcal{J}^S|$ is $(n-1-s)$-acyclic. 
Then, if $C(\cP)$ is $(n-1)$-acyclic, so is $C(\cH)$. 
\end{lem}
\begin{proof} 
Using Mayer-Vietoris one can determine the homology groups $\H_r(Z^S\star|\mathcal{J}^S|)$ in terms of $\H_i(Z^S)$ and 
$\H_j(|\mathcal{J}^S|)$ for $i+j<r$ so the hypothesis imply that the complex $Z^S\star|\mathcal{J}^S|$ is $(n-1-s+s-2+2)=(n-1)$-acyclic.

By equation~\eqref{complexS}, 
$$\H_j(D^{s+1}/D^s)=\bigoplus_{S\in\cP\setminus\cH,h(S)=s}\tilde{\H}_{j-1}(Z^S\star|\mathcal{J}^S|)\uparrow_{S}^G=0$$
for $0\leq j\leq n$. And this implies the result:
to see it  assume by induction that $D^{s+1}$ is $R$-$(n-1)$-acyclic (the beginning of the induction is $D^{h(\cP)}=C(\cP)$ 
which is $(n-1)$-acyclic) and consider the long exact homology sequence of the short exact sequence of complexes 
$0\to D^s\to D^{s+1}\to D^{s+1}/D^s\to 0$
$$\ldots\to0=\H_{i+1}(D^{s+1})\to\H_{i+1}(D^{s+1}/D^s)\to\H_i(D^s)\to\H_i(D^{s+1})=0\to\ldots$$
Thus also $D^s$ is $(n-1)$-acyclic. Eventually, $C(\cH)=D^0$ is $(n-1)$-acyclic. 
\end{proof}

\begin{obs}\label{obs:Morse} 
As noted by an anonymous referee of this paper, Lemma~\ref{lem:cosetposets} can also be proven using Morse theory, 
having $h_{\cH^c}$ as the Morse function. The complex $Z^S\star|\mathcal{J}^S|$ is the link of $S$ in $D^{s+1}$ so 
the hypothesis of Lemma~\ref{lem:cosetposets} is in fact a condition on the acyclicity of the link. The proof presented 
in this paper exhibits the fact from Morse theory that acyclicity of the links yields isomorphisms between the homology 
groups of the involved subcomplexes. 
\end{obs}

\begin{obs}\label{obs:R} 
In Lemma~\ref{lem:cosetposets} we can substitute the instances of ``being acyclic'' by ``being $R$ acyclic'' 
for any unital commutative ring $R$.
\end{obs}

\section{proof of Theorem~\ref{teo:mainsigma}}
As stated above,  our proof of Theorem~\ref{teo:mainsigma} is based in Theorem~\ref{teo:sigmacomplex}. To do that we need a suitable complex $X$.  

Let $A_\Gamma$ be an Artin group and consider the {\sl clique poset} 
$$\cP=\{A_\Delta\mid\Delta\subseteq\Gamma\text{ clique}\}$$
(recall that a {\sl clique} is a complete subgraph). If the Artin group $A_\Gamma$ is of $\FC$-type, 
then any clique of $\Gamma$ generates a finite type Artin subgroup so the coset complex $C_G(\cP)=C(\mathcal{P})$ 
of $\cP$ is the {\sl modified Deligne complex} for $A_\Gamma$ considered by Charney and Davis in \cite{Charney-kpi1}. 
In \cite{GodelleParis} the modified Deligne complex  is used to construct what is called the {\sl clique cube-complex} which is a CAT-(0) cube complex. 
 
For Artin groups of $\FC$-type, the modified Deligne complex was shown to be contractible in 
\cite{Charney-kpi1} but for completeness, we give a direct easy proof of the fact that the coset complex $C_G(\cP)$ 
is contractible in general, i.e., for arbitrary Artin groups $A_\Gamma$ possibly without the $\FC$ condition.

\begin{lem}\label{lem:complex} 
Let $G=A_\Gamma$ be an Artin group and consider $\cP$ the  clique poset. The coset complex $C_G(\cP)=C(\mathcal{P})$ is contractible.
\end{lem}
\begin{proof} 
We argue by induction on the number of vertices of $\Gamma$. Observe first that the result is obvious if $\Gamma$ is complete, 
because then $G$ itself lies in $\cP$. If $\Gamma$ is not complete we may decompose $\Gamma=\Gamma_1\cup\Gamma_2$ with 
$\Gamma_1,\Gamma_2\subsetneq\Gamma$ such that for $\Gamma_0=\Gamma_1\cap\Gamma_2$ no vertex in $\Gamma_1\setminus\Gamma_0$ 
is linked to any vertex in $\Gamma_2\setminus\Gamma_0$. This decomposition induces a decomposition of $G$ as the free product with amalgamation
$$G=A_{\Gamma_1}\star_{A_{\Gamma_0}}A_{\Gamma_2}.$$
Let  $C(\mathcal{P}_1)$, $C(\mathcal{P}_2)$ and  $C(\mathcal{P}_0)$ be the corresponding coset complexes for the clique posets of 
$A_{\Gamma_1}$, $A_{\Gamma_2}$ and $A_{\Gamma_0}$. By induction we may assume that they are all contractible. 
Consider the poset
$$\mathcal{BS}=\{gA_{\Gamma_1}\mid g\in G\}\cup\{gA_{\Gamma_2}\mid g\in G\}\cup\{gA_{\Gamma_0}\mid g\in G\}.$$
The geometric realization of this poset is precisely the barycentric subdivision of the Bass-Serre tree associated to the free amalgamated product above. Consider the map 
$$\begin{aligned}f:C(\mathcal{P})&\to\mathcal{BS}\\
gA_\Delta&\mapsto\left\{ 
\begin{aligned}
&gA_{\Gamma_0}\text{ if }\Delta\subseteq \Gamma_0\\
&gA_{\Gamma_1}\text{ if }\Delta\subseteq \Gamma_1, A_{\Delta}\not\leq A_{\Gamma_0}\\
&gA_{\Gamma_2}\text{ if }\Delta\subseteq \Gamma_2, A_{\Delta}\not\leq A_{\Gamma_0}.\\
\end{aligned}\right.
\end{aligned}
$$
Observe that each clique of $\Gamma$ is a subgraph either of $\Gamma_1$ or of $\Gamma_2$. 
Moreover, if $gS\subseteq gT$ with $S$ and $T$ both in $\cP$, then 
$f(gS)\subseteq f(sT)$ so it is a well-defined poset map and for any $gA_{\Gamma_i}\in\mathcal{BS}$, 
$$\{gS\in C(\mathcal{P})\mid f(gS)\leq gA_{\Gamma_i}\}=
gC_{A_{\Gamma_i}}(\mathcal{P}_i)$$
i.e., it is the coset poset of the clique poset of $A_{\Gamma_i}$ shifted by $g$. 
By induction, the posets $C_{A_{\Gamma_i}}(\cP_i)$ have contractible 
geometric realizations for $i=0,1,2$. So we may apply Quillen poset map Lemma (\cite{Benson}) and deduce that $f$ 
induces a homotopy equivalence between the geometric realizations. As the geometric realization $|\mathcal{BS}|$ 
is contractible, we deduce the same for the geometric realization $|C(\mathcal{P})|$.
\end{proof}

Note that the Artin group acts on the clique poset so we have a nice action on the geometric realization. However, 
this is not what we need to apply Theorem~\ref{teo:sigmacomplex} because the stabilizers of this action are not 
nice enough. In order to construct our suitable $X$, we will also need an auxiliary Lemma.

\begin{lem}\label{lem:spheres} 
Let $\Delta$ be a (non empty) complete graph with $s$ vertices and with $S:=A_\Delta$ even of $\FC$-type and consider the simplicial complex $Z^S$ with simplices
$$hS_0\subseteq hS_1\subseteq\dots \subseteq hS_r$$
for $h\in A_\Delta$ and each $S_j$ a special {\sl proper} subgroup of $A_\Delta$. Then $Z^S$ is $(s-2)$-acyclic.
\end{lem}

\begin{proof} We will denote $\Delta_v=\Delta\setminus \{v\}$. 
For each $g\in S$, $v\in\Delta$ let
$gX_{v}$ be the subcomplex of $Z^S$ induced by $\{gA_{\Omega}\mid \Omega\subseteq\Delta_v\}.$ Obviously, the set $\{gX_{v}\mid g\in S,v\in\Delta\}$ is a covering of $Z^S$ as
each special proper subgroup of $A_\Delta$ is inside one of the form $gA_{\Delta_v}$.  

Observe that for any $\Delta_1\subset\Delta$, if $g_1,g_2\in S$ are such that $g_1A_{\Delta_1}\neq g_2A_{\Delta_1}$, then 
$g_1A_{\Delta_1}\cap g_2A_{\Delta_1}=\emptyset$. And if we have $\Delta_1,\Delta_2$ and $g_1,g_2\in S$ so that $g_1A_{\Delta_1}\cap g_2A_{\Delta_2}\neq\emptyset$, then for some $x\in A_{\Delta_1}$, $y\in A_{\Delta_2}$, $g_1x=g_2y$ so 
$$g_1A_{\Delta_1}\cap g_2A_{\Delta_2}=g_1(A_{\Delta_1}\cap xA_{\Delta_2})=g_1A_{\Delta_1\cap\Delta_2}$$
because $g\in A_{\Delta_1}\cap xA_{\Delta_2}$ implies $x^{-1}g\in A_{\Delta_1}\cap A_{\Delta_2}=A_{\Delta_1\cap\Delta_2}$ ($A_{\Delta_1}\cap A_{\Delta_2}=A_{\Delta_1\cap\Delta_2}$ by \cite{Van}). This implies that any set of pairwise distinct $g_0X_{v_0},\ldots,g_kX_{v_k}$ has empty intersection if two of the $v_i$'s are equal and also that whenever such as set has non empty intersection, then the intersection is the subcomplex induced by 
$\{gA_{\Omega}\mid \Omega\subseteq\Delta_{v_1}\cap\ldots\cap\Delta_{v_k}\}$ which is contractible (it is the geometric realization of a poset with a maximal element). 
The last property implies that $Z^S$ is homotopy equivalent to the nerve of the covering (see \cite[Chap. VII Theorem 4.4]{BrownBook}) which we denote $N_\Delta$. 
The nerve has $k$-simplices the sets  $\{g_0X_{v_0},\ldots,g_kX_{v_k}\}$ with non empty intersection and the discussion above implies that the cardinality of such as set can be at most $s$ so this nerve  is a complex of dimension $s-1$. 
To deduce that  this complex is $(s-2)$-connected one can use \cite[(4.18), Theorem (2.15)]{Deligne-immeubles}, but it is also possible to give a direct argument as follows.
As we are assuming that $S$ is even of $\FC$-type, we have a decomposition
$$S=S_1\times\ldots\times S_t$$
where each $S_i=A_{\Delta_i}$ is either infinite cyclic or Artin dihedral (of even type). This means that for each $g\in S$ and $v\in\Delta$, $gA_{\Delta_v}=g_iA_{\Delta_v}$ for some $g_i\in S_i$ where $i$ is the index such that $v\in\Delta_i$. From this one deduces that $N_\Delta$ is the join
$$N_\Delta=N_{\Delta_1}\star\ldots\star N_{\Delta_t}$$
where each $N_{\Delta_i}$ is the nerve associated to the same construction but for $S_i$. Moreover, if $S_i$ is infinite cyclic, obviously $N_{S_i}$ is a disjoint union of points, thus non empty and if $S_i$ is Artin dihedral, then it is easy to see that $N_{S_i}$ is 0-connected. As $s=a+2b$ where $a$ is the number of $S_i$'s which are infinite cyclic, we get the result.
 \end{proof}.

At this point, we are able to construct the desired $X$. Note that Theorem~\ref{teo:sigmacomplex} together with the next result imply Theorem~\ref{teo:mainsigma}.

In the introduction we have defined the subposet $\mathcal{B}^\chi$ of $\mathcal{P}$ as those $A_\Delta$ for $\Delta\subseteq\Gamma$ clique
such that for each vertex $v$ in $\Delta$ either $v$ is dead or $v\in e$ for $e$ a dead edge in $\Delta$, this includes the case when $\Delta=\emptyset$. The hypothesis that $A_\Gamma$ is even of $\FC$-type implies that $A_\Delta$ is a direct product of dihedral groups (corresponding to edges with label bigger than 2) and of infinite cyclic groups (generated by dead vertices). 
Taking into account that the center of the dihedral group generated by, say, $x$ and $y$ is the infinite cyclic group generated by $xy$ we see that for such a $\Delta$ we have  $\chi(Z(A_\Delta))=0$. The converse is also obvious. So we have
$$\mathcal{B}^\chi=\{A_\Delta\in\mathcal{P}\mid \chi(Z(A_\Delta))=0\}.$$
Denote 
$$\mathcal{H}^\chi:=\mathcal{P}\setminus\mathcal{B}^\chi,$$
one has the following result on the homotopy of the geometric realization of the coset poset $X:=|C\mathcal{H}^\chi|$.

\begin{prop}
Let $A_{\Gamma}$ be an even Artin group of $\FC$-type. Let $\chi:A_\Gamma\to\RR$ be a character. If the strong $n$-link condition holds, then the geometric realization of the coset poset $X:=|C\mathcal{H}^\chi|$ is $(n-1)$-acyclic. 
\end{prop}

\begin{proof} 
Use Lemma~\ref{lem:cosetposets} for $\cP$ the clique poset with $h(A_\Delta)=|\Delta|$. Fix 
$S=A_\Delta\in\mathcal{B}^\chi=\mathcal{P}\setminus\mathcal{H}^\chi$ with $h(S)=s$. The complex denoted $Z^S$ in Lemma~\ref{lem:cosetposets}
  is the simplicial complex of Lemma~\ref{lem:spheres} and $\mathcal{J}^S$ is the  poset
 $$\mathcal{J}^S:=\{T\in\cP\mid S\subseteq T,\chi(Z(T))\neq 0\}.$$ 
 Now, consider the poset 
   $$\mathcal{L}^S:=\{L=A_{\sigma}\in\cP\mid \emptyset\neq\sigma\text{ clique of }{\lk}_{\LL^\chi}(\Delta)\}.$$
Let $A_\sigma\in\mathcal{L}^S$. Then for $T=A_{(\sigma\star\Delta)}$, we have $S\leq T$ and $T\in\cP$. We claim that $T\in\mathcal{J}^S$. To see it, note that as $\sigma\neq\emptyset$ and $\sigma\subseteq\LL^\chi$, $Z(A_\sigma)\neq 0$. As $\Gamma$ is even of $\FC$-type, this implies that either there is some $v\in\sigma$, $v\in Z(A_\sigma)$ with $\chi(v)\neq 0$ or there are $v,w\in\sigma$, $(vw)^k\in Z(A_\sigma)$ for some $k$ with $\chi(v)+\chi(w)\neq 0$. Again, the condition that $\Gamma$ is of $\FC$-type implies that in the second case, $(vw)^k\in Z(T)$ so $\chi(Z(T))\neq 0$. In the first case, either $v\in Z(T)$ so again $\chi(Z(T))\neq 0$ or there is some $w\in\Delta$ with $(vw)^k\in Z(T)$ for some $k$. In this case moreover $w\in Z(S)$ thus $\chi(w)=0$. Therefore $\chi((vw)^k)\neq 0$ and again $\chi(Z(T))\neq 0$. The claim follows and therefore
we have a well defined poset map
   $$
\begin{aligned}
f: \mathcal{L}^S&\to\mathcal{J}^S\\
 A_\sigma&\mapsto A_{(\sigma\star\Delta)}.\\
 \end{aligned}$$
 We claim that this map induces a homotopy equivalence between the corresponding geometric realizations. To see it, let $T\in\mathcal{J}^S$ and consider
 $$f^{-1}_{\leq T}=\{U\in \mathcal{L}^S\mid f(U)\leq T\}.$$
 By Quillen's poset map Theorem (see~\cite{Benson}) it suffices to check that the poset $f^{-1}_{\leq T}$ has contractible geometric realization. Put $T=A_\nu$. Then $\nu$ is a clique of $\Gamma$ 
 such that $\nu=\Delta\star\sigma$ for some $\emptyset\neq\sigma$ clique in $\lk_\Gamma(\Delta)$. We can describe $\sigma$ as a join
 $$\sigma=e_1\star\dots\star e_t\star p_1\star\dots p_s$$
 where each $e_i$ is a single edge with label $>2$ and each $p_i$ is a single point and all the edges not in some $e_i$ are labeled by 2.
 We may order them so that $\chi(e_1),\ldots,\chi(e_l)=0$, $\chi(e_{l+1}),\ldots,\chi(e_t)\neq 0$, $\chi(p_1),\ldots,\chi(p_r)=0$, $\chi_(p_{r+1}),\ldots,\chi(p_s)\neq 0$. Then 
  $$\sigma\cap\LL^\chi=w_1\star\dots\star w_l\star e_{l+1}\star\dots \star e_t\star p_{r+1}\star\dots\star p_s$$
 where each $w_i$ is the disconnected set consisting of the two vertices of each $e_i$ has as barycentric subdivision precisely the geometric realization of the poset $f^{-1}_{\leq T}$. As $e_{l+1}\star\dots \star e_t\star p_{r+1}\star\dots\star p_s$ is either contractible or empty, so show that $f^{-1}_{\leq T}$ is contractible we only have to show that $e_{l+1}\star\dots \star e_t\star p_{r+1}\star\dots\star p_s$ is not empty. 
 As $T\in\mathcal{J}^S$, $\chi(Z(T))\neq 0$. If there is some $v$ vertex of $\nu$ with $v\in Z(T)$ and $\chi(v)\neq 0$ then the fact that $\chi(Z(S))=0$ implies $v\in\sigma$ so $v$ belongs to $\{p_{r+1},\ldots,p_s\}$. So we are left with the case when $\chi(vw)\neq 0$ for $v,w$ vertices of an edge of $\nu$ with label $>2$. If, say, $v$ lies in $\Delta$, then $v\in Z(S)$ so $\chi(v)=0$ and we deduce $\chi(w)\neq 0$. Moreover, in this case the $\FC$-condition implies that $w$ is in the center of $A_\sigma$, i.e., $w$ belongs to $\{p_{r+1},\ldots,p_s\}$. So we may assume that both $v,w$ lie in $\sigma$ so the edge joining them lies in the set $\{e_{l+1},\ldots,e_t\}$.
\end{proof}

We finish this section with a couple of example to illustrate how to apply Theorem~\ref{teo:mainsigma}.

\begin{exam}
Let $\Gamma$ be the graph and $\chi$ the character
$$
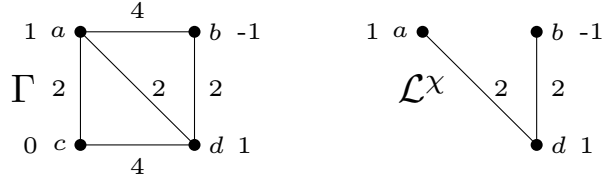
\begin{tikzpicture}[scale=1.5, transform shape]
\tikzstyle{subj} = [circle, minimum width=3pt, fill, inner sep=0pt]
\tikzstyle{obj}  = [circle, minimum width=3pt, draw, inner sep=0pt]
\node[subj] (n1) at (1,1) {};
\node[subj] (n2) at (2,1) {};
\node[subj] (n3) at (2,2) {};
\node[subj] (n4) at (1,2) {};
\draw (0.3,1.5) node {{\small{$\Gamma$}}};
\draw (1,2) node[left] {\tiny{$a$}} ;
\draw (0.75,2) node[left] {\tiny 1} ;
\draw (0.75,1) node[left] {\tiny 0} ;
\draw (2.25,2) node[right] {\tiny -1} ;
\draw (2.25,1) node[right] {\tiny1} ;
\draw (2,2) node[right] {\tiny{$b$}};
\draw (1,1) node[left] {\tiny{$c$}} ;
\draw (2,1) node[right] {\tiny{$d$}} ;
\draw(1.5,2) node[above]{\tiny{4}};
\draw(1.5,1) node[below]{\tiny{4}};
\draw(1,1.5) node[left]{\tiny{2}};
\draw(2,1.5) node[right]{\tiny{2}};
\draw(1.5,1.5) node[right]{\tiny{2}};
\draw (1,1) -- (1,2) -- (2,2) -- (2,1)--(1,1);
\draw (1,2)--(2,1);
\draw (3.8,1.5) node {{\small{$\LL^\chi$}}};
\node[subj] (n2) at (2+3,1) {};
\node[subj] (n3) at (2+3,2) {};
\node[subj] (n4) at (1+3,2) {};
\draw (1+3,2) node[left] {\tiny{$a$}} ;
\draw (0.75+3,2) node[left] {\tiny 1} ;
\draw (2.25+3,2) node[right] {\tiny -1} ;
\draw (2.25+3,1) node[right] {\tiny 1} ;
\draw (2+3,2) node[right] {\tiny{$b$}};
\draw (2+3,1) node[right] {\tiny{$d$}} ;
\draw(2+3,1.5) node[right]{\tiny{2}};
\draw(1.5+3,1.5) node[right]{\tiny{2}};
\draw (2+3,2) -- (2+3,1);
\draw (1+3,2)--(2+3,1);
\end{tikzpicture} 
$$
For $\Delta=(ab)$, $Z(A_\Delta)$ is generated by $ab$ so $\chi(Z(A_\Delta))=0$.

For $\Delta=(a,b,d)$, $Z(A_\Delta)$ is generated by $ab$, $d$ so $\chi(Z(A_\Delta))\neq 0$. 

We get: $\mathcal{P}\setminus{\mathcal{H}_\chi}=\{\emptyset,(c),(ab)\}$. The links are:

$$
\begin{tikzpicture}[scale=1.5, transform shape]
\tikzstyle{subj} = [circle, minimum width=2pt, fill, inner sep=0pt]
\tikzstyle{obj}  = [circle, minimum width=2pt, draw, inner sep=0pt]
\draw(-2,0) node{\tiny{$\lk_{\LL^\chi}(\emptyset)=\LL^\chi$}};
\draw(-0.2,0) node[left=-15] {\tiny{$\lk_{\LL^\chi}(c)=$}};
\draw(0.5,0.15)--(0.8,-0.15);
\node[subj] (n1) at (0.5,0.15) {};
\node[subj] (n2) at (0.8,-0.15){};
\draw(0.8,-0.15) node[right]{\tiny{$d$}};
\draw(0.5,0.15) node[left]{\tiny{$a$}};
\draw(2,0) node{\tiny{$\lk_{\LL^\chi}(ab)=d$}};
\end{tikzpicture} 
$$
All the links are contractible so $\chi\in\Sigma^\infty(A_\Gamma,\ZZ)$.
\end{exam}

\begin{exam}
Let $\Gamma$ be the graph  and $\chi$ the character
$$
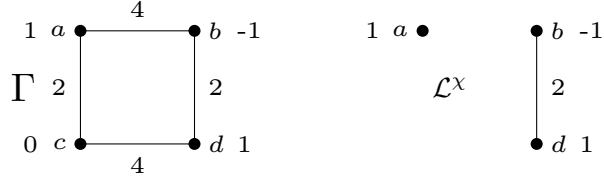
\begin{tikzpicture}[scale=1.5, transform shape]
\tikzstyle{subj} = [circle, minimum width=3pt, fill, inner sep=0pt]
\tikzstyle{obj}  = [circle, minimum width=3pt, draw, inner sep=0pt]
\node[subj] (n1) at (1,1) {};
\node[subj] (n2) at (2,1) {};
\node[subj] (n3) at (2,2) {};
\node[subj] (n4) at (1,2) {};
\draw (0.3,1.5) node {{\small{$\Gamma$}}};
\draw (1,2) node[left] {\tiny{$a$}} ;
\draw (0.75,2) node[left] {\tiny 1} ;
\draw (0.75,1) node[left] {\tiny 0} ;
\draw (2.25,2) node[right] {\tiny -1} ;
\draw (2.25,1) node[right] {\tiny1} ;
\draw (2,2) node[right] {\tiny{$b$}};
\draw (1,1) node[left] {\tiny{$c$}} ;
\draw (2,1) node[right] {\tiny{$d$}} ;
\draw(1.5,2) node[above]{\tiny{4}};
\draw(1.5,1) node[below]{\tiny{4}};
\draw(1,1.5) node[left]{\tiny{2}};
\draw(2,1.5) node[right]{\tiny{2}};
\draw (1,1) -- (1,2) -- (2,2) -- (2,1)--(1,1);
\draw (3.8,1.5) node[scale=.7,right] {$\mathcal L^\chi$};
\node[subj] (n2) at (2+3,1) {};
\node[subj] (n3) at (2+3,2) {};
\node[subj] (n4) at (1+3,2) {};
\draw (1+3,2) node[left] {\tiny{$a$}} ;
\draw (0.75+3,2) node[left] {\tiny 1} ;
\draw (2.25+3,2) node[right] {\tiny -1} ;
\draw (2.25+3,1) node[right] {\tiny 1} ;
\draw (2+3,2) node[right] {\tiny{$b$}};
\draw (2+3,1) node[right] {\tiny{$d$}} ;
\draw(2+3,1.5) node[right]{\tiny{2}};
\draw (2+3,2) -- (2+3,1);
\end{tikzpicture} 
$$
As before: $\mathcal{P}\setminus{\mathcal{H}_\chi}=\{\emptyset,(c),(ab)\}$. However, 
$\lk_{\LL^\chi}(\emptyset)=\LL^\chi$ is not connected, 
so $[\chi]$ might not even be in $\Sigma^1(A_\Gamma,\ZZ)$.

Incidentally, in this case, the hypothesis of Corollary~\ref{corol:sigmachar} below is satisfied for $p=2$,
hence  the converse of Theorem~\ref{teo:mainsigma} also holds true so
$$[\chi]\not\in\Sigma^1(A_\Gamma,\ZZ).$$
\end{exam}

\section{The free part of the homology groups of Artin kernels}\label{sec:free}

Let $A_\Gamma$ be an even Artin group of $\FC$-type  and $\chi:A_\Gamma\to\ZZ$ a discrete character. In this section, 
we are interested in the homology groups $\H_n(A_\Gamma^\chi,\FF)$ where $\FF$ is a field of characteristic $p$ (either zero or a prime). 
More precisely, we want to characterize when they are finite $\FF$ dimensional and, more generally, to compute their 
free part when seen as $\FF[t^{\pm 1}]$-modules via~$\chi$.

To do that we first develop a $p$-local version of some of the notions that we used above.

We say that an edge $e$ of $\Gamma$ with label $2\tilde{\ell}_e$ is {\sl $p$-dead} if $m_e=\chi(v)+\chi(w)=0$ and $p\mid \tilde{\ell}$. 
In \cite{Blasco-conmar-ji-Homology}, $p$-dead edges were called $\FF$-resonant (recall that $\FF$ is a field of characteristic $p$).

We set $\LL^\chi_p$ for the subgraph of $\Gamma$ that we get when we remove dead vertices and open $p$-dead edges. 
This notation is consistent with $\LL^\chi_0$ because no edge can be $0$-dead. Note that the set of dead edges is the union of 
the sets of $p$-dead edges where $p$ runs over all prime numbers and therefore $\LL^\chi=\bigcap_{p\textrm{ prime}}\LL_p^\chi$. 

Let $\mathcal{B}_p^\chi$ be the set of those $A_{\Delta}\in\cP$ for a clique $\Delta\subseteq\Gamma$ such that
for each vertex $v$ in $\Delta$ either $v$ is dead or $v\in e$ for $e$ a dead edge in $\Delta$ (this includes the case 
when $\Delta=\emptyset$). Note that as edges which are $p$-dead for some $p$ are dead, this condition implies 
that~$\chi(Z(A_{\Delta}))=0$. 

\begin{dfn} Assume that for any 
$A_\Delta\in\mathcal{B}_p^\chi$ with $|\Delta|\leq n$ the link $\hat{\lk}_{\LL_p^\chi}(\Delta)$ 
is $p-(n-1-|\Delta|)$-acyclic, meaning that its homology up to degree $(n-1-|\Delta|)$ with coefficients in a field of characteristic 
$p$ vanishes. Then we say that $\chi$ satisfies \emph{the strong $p$-$n$-link condition}.
\end{dfn}

The homology groups $\H_n(A_\Gamma^\chi,\FF)$ are precisely the homology groups of the  $\FF$-chain complex 
$C_n(\Sal^\chi_\Gamma)$ described in \cite[Section 2]{Blasco-conmar-ji-Homology}. 
This complex was obtained using the $\chi$-cyclic cover of the Salvetti
complex of $A_\Gamma$ (see~\cite{Salvetti-topology,Charney-finite,Paris-lectures}) and has
$$C_n(\Sal^\chi_\Gamma)=\FF[t^{\pm 1}]\otimes_\FF \bar{C}(\hat\Gamma)_{n-1}$$ 
where $\bar{C}(\hat\Gamma)_{n-1}$ is the augmented chain complex of the flag complex $\hat\Gamma$ shifted by one. 
The differential of $C_n(\Sal^\chi_\Gamma)$ can be described as follows (see \cite{Blasco-conmar-ji-Homology}, after Remark 2.3). 
For each edge $e\in E_\Gamma$ let $2\tilde{\ell}_e$ be its label in $\Gamma$ and denote $q_k(x)=(x^k-1)/(x-1)$. 
Then for any $X\subseteq\Gamma$ complete we have
\begin{equation}\label{eq:differential}\partial_n^\chi\sigma^\chi_X=\sum_{v\in X}\langle X_v\mid X\rangle b_{v,X}\sigma^\chi_{X_v}\end{equation}
where we are denoting $X_v$ the clique obtained from $X$ by removing $v$ and
\begin{equation}\label{eq:coefficient}b_{v,X}:=(t^{m_v}-1)
\prod_{{\tiny{\array{c}w\in X_v\\e=\{v,w\}\in E_\Gamma\endarray}}}q_{\tilde{\ell}_e}(t^{m_e}).
\end{equation}

In particular, if $\tilde{\ell}_e=1$, then $q_{\tilde{\ell}_e}(t^{m_e})=1$ and if $m_e=0$, $q_{\tilde{\ell}_e}(t^{m_e})=\tilde{\ell}_e$.
Recall that an edge $e\in E_\Gamma$ is called $p$-dead if  $m_e=0$ and $p\mid \tilde{\ell}_e$; otherwise will be called {\sl $p$-living}.
So we see that in~\eqref{eq:differential}, the coefficient $b_{v,X}$ vanishes if either $v$ is dead or belongs to a $p$-dead edge in $X$.

Let $I$ be the augmentation ideal of the ring $R=\FF[t^{\pm 1}]$, i.e., the kernel of the augmentation map $R\to \FF$ with $t\mapsto 1$, 
$R$ is a principal ideal domain and $I$ is the ideal generated by $t-1$. Since $I$ is a prime ideal, we can localize and get a new ring $R_I$. 
We can also localize the complex $C_n(\Sal^\chi_\Gamma)$ and get a new complex $C_n(\Sal^\chi_\Gamma)_I$ with $n$-term
$$C_n(\Sal^\chi_\Gamma)_I=R_I\otimes_\FF \bar{C}(\hat\Gamma)_{n-1}$$ 
whose differential we also denote by $\partial_n^\chi$. Since localizing is flat, the $R$-free part of the homology of 
$(R\otimes_\FF \bar{C}(\hat\Gamma)_{n-1},\partial_n^\chi)$ has the same rank as the $R_I$-free part of the homology 
of $(R_I\otimes_\FF \bar{C}(\hat\Gamma)_{n-1},\partial_n^\chi)$. 
But in this complex we can normalize over the living vertices and $p$-living edges in the following way. 

Let $X\subseteq\Gamma$ be a clique and put
$$a_X=\prod_{v\in X\textrm{ living}}(t^{m_v}-1)\prod_{e\in E_X\textrm{ $p$-living}}q_{\tilde{\ell}_e(t^{m_e})}.$$
Let $\mu_X$ be the multiplicity of $t-1$ as a factor of $a_X$. Then 
$$a_X=(t-1)^{\mu_X}h_X$$
where $h_X$ is a unit in our ring $R_I$.
Observe that for any $v\in X$ we have
\begin{equation}\label{eq:Xv}(t-1)^{\mu_X}h_X=a_X=b_{v,X}a_{X_v}=b_{v,X}(t-1)^{\mu_{X_v}}h_{X_v}.\end{equation}

We can choose an integer $k$ such that $k|X|\geq\mu_X$ for any $X\subseteq\Gamma$ clique. Let $X\subseteq\Gamma$ be a clique and set

$$\tilde{\sigma}_X:=(t-1)^{k|X|-\mu_X}\frac{1}{h_X}\sigma_X.$$
Then 
$$\partial_n^\chi(\tilde{\sigma}_X)=(t-1)^{k|X|-\mu_X}\frac{1}{h_X}
\partial_n^\chi(\sigma_X)=(t-1)^{k|X|-\mu_X}\frac{1}{h_X}\sum_{v\in X}\langle X_v\mid X\rangle b_{v,X}\sigma^\chi_{X_v}.$$

Recall that the summand associated to each $v\in X$ vanishes if either $v$ is dead or it belongs to a $p$-dead edge in $X$.
Otherwise, using~\eqref{eq:Xv} we see that that summand is, up to a sign,
$$(t-1)^{k|X|-\mu_{X}}\frac{b_{v,X}}{h_X}\sigma^\chi_{X_v}=(t-1)^{k|X|-\mu_{X_v}}\frac{1}{h_{X_v}}
\sigma^\chi_{X_v}=(t-1)^k(t-1)^{k|X_v|-\mu_{X_v}}\frac{1}{h_{X_v}}\sigma^\chi_{X_v}={\tilde\sigma}^\chi_{X_v}.$$

Hence, if we denote by $\mathcal{F}^\chi_p(X)$ the subgraph that we get from $X$ when we remove dead vertices and {\sl closed} $p$-dead edges we have
\begin{equation}
\label{eq:partial}
\partial_n^\chi(\tilde{\sigma}_X)=
(t-1)^k\sum_{v\in\mathcal{F}^\chi_p(X)}\langle X_v\mid X\rangle\tilde{\sigma}_{X_v}.
\end{equation}
Observe that $\mathcal{F}^\chi_p(X)\subseteq X\cap\LL_p^\chi$ where  $\LL^\chi_p$ is the $p$-living subgraph  defined in the 
introduction, i.e.,  the subgraph of $\Gamma$ that we get when we remove dead vertices and {\sl open} $p$-dead edges. 

Now, for each $n$ let $\tilde C_n(\Sal^\chi_\Gamma)_I$ be the sub $R_I$-module of $C_n(\Sal^\chi_\Gamma)_I$ generated by the $\tilde\sigma_X$, $|X|=n$. The computations above imply that $(\tilde C_n(\Sal^\chi_\Gamma)_I,\partial_n^\chi)$ is a subcomplex of  $(C_n(\Sal^\chi_\Gamma)_I,\partial_n^\chi)$ and by definition each quotient $C_n(\Sal^\chi_\Gamma)_I/\tilde C_n(\Sal^\chi_\Gamma)_I$ is a  $R_I$-torsion module. Using the long exact homology sequence we see that the $R_I$-free part of the homology of $(C_n(\Sal^\chi_\Gamma)_I,\partial_n^\chi)$ equals the $R_I$-free part of the homology of $(\tilde C_n(\Sal^\chi_\Gamma)_I,\partial_n^\chi)$. So from now on we consider this last complex. 

We define a new map $d_n^\chi:\tilde C_n(\Sal^\chi_\Gamma)_I\to\tilde C_{n-1}(\Sal^\chi_\Gamma)_I$ by
$$d_n^\chi={\frac{1}{(t-1)^k}}\partial_n^\chi.$$

\begin{lem}\label{lem:tecd} With the notation above we have
\begin{enumerate}[label*=$\roman*)$]
\item
\label{lem:tecd-1} 
$\ker\partial_n^\chi=\ker d_n^\chi$,

\item
\label{lem:tecd-2} 
$\im\partial^\chi_n\subseteq\im d^\chi_n$,

\item
\label{lem:tecd-3} 
$\im d_n^\chi\cap I^k\tilde C_n(\Sal^\chi_\Gamma)_I=\im\partial_n^\chi,$

\item
\label{lem:tecd-4}
$\dim_\FF(\im d_n^\chi/\im\partial_n^\chi)<\infty$.
\end{enumerate}
\end{lem}
\begin{proof} 
\ref{lem:tecd-1} is obvious. For~\ref{lem:tecd-2}, take $a\in\im\partial_n^\chi$. 
Then $a=(t-1)a_1$ and $a=\partial^\chi_n(b)$ so $a_1=d_n^\chi(b)\in\im d_n^\chi$ so $a=(t-1)a_1\in\im d_n^\chi$. 
For~\ref{lem:tecd-3}, the fact that $\im\partial_n^\chi\subseteq\im d_n^\chi\cap I^k\tilde C_n(\Sal^\chi_\Gamma)_I$ 
is obvious because of~\ref{lem:tecd-2}. Conversely, take $a\in\im d_n^\chi\cap I^k\tilde C_n(\Sal^\chi_\Gamma)_I$. 
Then $a=d_n^\chi(b)$ and the fact that $a\in I^k\tilde C_n(\Sal^\chi_\Gamma)_I$ together with the definition of 
$d_n^\chi$ implies that also $b\in I^k\tilde C_n(\Sal^\chi_\Gamma)_I$ so $b=(t-1)^kb_1$ and 
$\partial_n^\chi(b_1)={\frac{1}{(t-1)^k}}\partial_n^\chi(b)=d_n^\chi(b)=a$ thus $a\in\im\partial_n^\chi$. Finally,~\ref{lem:tecd-4} follows from~\ref{lem:tecd-3}.
\end{proof}

\begin{prop} 
For each $n$, $\dim_\FF\H_n(A_\Gamma^\chi,\FF)<\infty$ if and only if the $n$-th homology of the localized chain complex 
$I^k\tilde C_n(\Sal^\chi_\Gamma)_I$ respect to $d_\bullet^\chi$ has finite $\FF$-dimension, i.e., if and only if 
$\dim_\FF\ker d_n^\chi/\im d_{n+1}^\chi<\infty$. 
\end{prop}
\begin{proof} Note that for each $n$ there is a short exact sequence
$$0\to \im d_{n+1}^\chi/\im\partial_{n+1}^\chi\to 
\ker\partial_n^\chi/\im\partial_{n+1}^\chi\to \ker\partial_n^\chi/\im d_{n+1}^\chi\to 0.$$
Since the left-hand side is of finite $\FF$-dimension by Lemma~\ref{lem:tecd}\ref{lem:tecd-4} and 
$\ker\partial_n^\chi=\ker d_n^\chi$ by Lemma~\ref{lem:tecd}\ref{lem:tecd-1}, the result follows. 
\end{proof}

\begin{prop}\label{freehomology}
$$\ker d_n^\chi/\im d_n^\chi=R_I\otimes_\FF\bigoplus_{A_X\in\mathcal{B}_p^\chi,|X|\leq n}
\overline{\H}_{n-1-|X|}(\hat\lk_{\LL^\chi_p}(X)).$$
\end{prop}
\begin{proof} From~\eqref{eq:partial} we have
$$d_n^\chi(\tilde{\sigma}_X)=\sum_{v\in \mathcal{F}^\chi_p(X)}\langle X_v\mid X\rangle\tilde{\sigma}_{X_v}.$$

Let $\emptyset\neq X\subseteq\Gamma$ be a clique. Let $\mathcal{B}_p^\chi(X)=Y$ be the subgraph of $X$ 
generated by dead vertices and closed $p$-dead edges and $Z=\mathcal{F}^\chi_p(X)$. Then any vertex of 
$X$ lies either in $Y$ or in $Z$, in other words, $X$ is the subgraph generated by $Y\cup Z$. 
Note that $Z\subseteq\lk_{\LL^\chi_p}(Y)$ is a clique and $A_Y\in\mathcal{B}_p^\chi$, obviously $Y$ is 
the biggest subgraph of $X$ satisfying this.

Conversely, given $A_Y\in\mathcal{B}_p^\chi$ and a clique $Z\subseteq\lk_{\LL^\chi_p}(Y)$, 
then the subgraph $X$ of $\Gamma$ generated by $Y\cup Z$ is a clique. We claim that $Y=\mathcal{B}_p^\chi(X)$, 
obviously $Y\subseteq\mathcal{B}_p^\chi(X)$. If there is some $v\in\mathcal{B}_p^\chi(X)$, $v\not\in Y$, then $v\in Z$ so 
it can not be dead and there must be some $p$-dead edge $e\in X$ with $e=(v,w)$. As $Z$ is a clique we cannot have $w\in Z$ 
so $w\in Y$. Then $0=m_e=m_v+m_w$ so $m_w\neq 0$, in other words, $w$ is not a dead vertex and as $A_Y\in\mathcal{B}_p^\chi$, 
we deduce that there must be some $p$-dead edge $e_1\in Y$ with $e_1=(w,u)$ for some other $u\in Y$. But then observe that the 
vertices $v,u,w$ from a triangle in $\Gamma$ and the fact that both $e$ and $e_1$ are $p$ dead implies that both have labels 
bigger than 2 which contradicts the $\FC$-condition. Moreover we also deduce that $Z=\mathcal{F}^\chi_p(X)$.

We will check that for each $A_Y\in\mathcal{B}_p^\chi$ there is a subcomplex $(D_Y)_\bullet$ of 
$(R_I\otimes_\FF \bar{C}(\hat\Gamma)_\bullet,d_\bullet^\chi)$ so that, as complexes,
$$\tilde C_\bullet(\Sal^\chi_\Gamma)_I=\bigoplus_{A_Y\in\mathcal{B}_p^\chi}(D_Y)_\bullet.$$
To see it, let $(D_Y)_k=0$ for $0\leq k\leq |Y|-1$ and for $k\geq|Y|$,
$$(D_Y)_k=\oplus\{R_I\tilde\sigma_X\mid |X|=k, X\subseteq\Gamma\text{ clique}, 
Y= \mathcal{B}_p^\chi(X)\}.$$
The fact that this is a $d_\bullet^\chi$-subcomplex follows from the fact that  for $\tilde{\sigma}_X\in (D_Y)_n$, $d_n^\chi(\tilde{\sigma}_X)$
vanishes in all the summands 
not in $(D_Y)_{n-1}$, more explicitly:
$$d_n^\chi(\tilde{\sigma}_X)=\sum_{v\in \mathcal{F}^\chi_p(X)}\langle X_v\mid X\rangle\tilde{\sigma}_{X_v}$$ 
and as $v\in \mathcal{F}^\chi_p(X)$, $\tilde{\sigma}_{X_v}\in (D_Y)_{n-1}$.

Moreover, the discussion above implies that we can identify
$$(D_Y)_k=R_I\otimes\overline{C}_{k-|Y|-1}(\hat\lk_{\LL^\chi_p}(Y))$$
and the fact that each $X$ determines uniquely $Y=\mathcal{B}_p^\chi(X)$ implies that
$$R_I\otimes_\FF \bar{C}(\hat\Gamma)_\bullet=\bigoplus_{A_Y\in\mathcal{B}_p^\chi}R_I\otimes\overline{C}_{\bullet+1+|Y|}(\hat\lk_{\LL^\chi_p}(Y)).$$
Therefore the result follows.
\end{proof}

As a consequence, we obtain the following result.

\begin{teo}\label{teo:free} Let $G=A_\Gamma$ be an even Artin group of $\FC$-type, $\chi:G\to\ZZ$ a discrete character with 
kernel $A_\Gamma^\chi$ and $\FF$ a field of characteristic $p$. Then the free part of the homology groups $\H_n(A_\Gamma^\chi,\FF)$ 
seen as $\FF[t^{\pm1}]$-modules has rank 
$$\sum_{A_X\in\mathcal{B}_p^\chi,|X|\leq n}
\dim_\FF\overline{\H}_{n-1-|X|}(\hat\lk_{\LL^\chi_p}(X),\FF).$$
\end{teo}

Therefore,
\begin{corol} 
Let $G=A_\Gamma$ be an even Artin group of $\FC$-type, $\chi:G\to\RR$ a character and 
$\FF$ a field of characteristic $p$. Then 
$\dim_\FF\H_i(A^\chi_\Gamma,\FF)<\infty$ for $0\leq i\leq n$ if and only if $\chi$ satisfies the strong $p$-$n$-link condition. 
\end{corol}

In the particular case $p=0$, note that $\mathcal{B}_0^\chi$ is just the set of those $A_X\in\cP$ with $X\subseteq\Gamma\setminus\LL^\chi_0$.

We also deduce a partial converse to Theorem~\ref{teo:mainsigma}.

\begin{corol}\label{corol:sigmachar} 
Let $G=A_\Gamma$ be an even Artin group of $\FC$-type, and $0\neq\chi:G\to\RR$ be a character 
such that $\LL^\chi_p=\LL^\chi$ for some $p$ either zero or prime. Assume that the strong $p$-$n$-link condition 
fails for $\chi$. Then $[\chi]\not\in\Sigma^n(G,\ZZ)$.
\end{corol}
\begin{proof} Let $\chi$ be a character that does not satisfy the strong $p$-$n$-link condition. Assume first that $\chi$ is discrete, i.e., 
$\chi(G)\subseteq\ZZ$. Let $\FF$ be a field of characteristic $p$. By Proposition~\ref{freehomology} and the discussion 
above we deduce that some of the homology groups $\H_i(A_\Gamma^\chi,\FF)$ has infinite dimension as an $\FF$-vector space, 
thus $A_\Gamma^\chi$ is not of type $\FP_n$ thus $[\chi]\not\in\Sigma^n(G,\ZZ)$. For the general case, i.e., 
when $\chi:G\to\RR$ is not necessarily discrete, consider the set
$$\{[\varphi]\mid\varphi:G\to\ZZ,\LL^\varphi=\LL^\chi\}.$$
Observe that $[\chi]$ lies in the closure of this set. The discrete case considered above implies 
$$\{[\varphi]\mid\varphi:G\to\ZZ,\LL^\varphi=\LL^\chi_0\}\subseteq\Sigma^c(G,\ZZ)$$
and as $\Sigma^c(G,\ZZ)$ is closed we deduce that also $[\chi]\in\Sigma^c(G,\ZZ)$.
\end{proof} 

\begin{exam} Let $G=\D_{\ell}$ be the dihedral Artin group associated to a graph $\Gamma$ which consists of a single edge 
$e$ with vertices $v,w$ and label $\ell=2\tilde{\ell}$ and let $\chi:G\to\ZZ$ given by $\chi(v)=1$, $\chi(w)=-1$. Let $\FF$ be a 
field of characteristic $p$. In this case the homology groups $\H_n(A_\Gamma^\chi,\FF)$ vanish for $n>1$ and one can compute directly 
$\H_1(A_\Gamma^\chi,\FF)$ for a field $\FF$ using the description of the differential~\eqref{eq:differential} above and gets:
$$
H_1(A_\Gamma^\chi;\FF)=
\begin{cases}
\FF[t^{\pm 1}] & \text{ if } p|\tilde{\ell}\\
\frac{\FF[t^{\pm 1}]}{(t-1)} & \text{ otherwise.}
\end{cases}
$$
This is precisely what Theorem~\ref{teo:free} predicts: if $p\nmid\tilde{\ell}$, there are no $p$-dead edges and no dead vertices 
which means $\mathcal{B}_p^\chi=\{1\}$ and $\LL_p^\chi=\Gamma$. The link $\lk_{\LL^\chi_p}(\emptyset)$ is the whole $\Gamma$ so the 
associated flag complex is contractible and the associated reduced homology groups vanish. By contrast, if $p|\tilde{\ell}$, the 
edge $(v,w)$ is $p$-dead so $\mathcal{B}_p^\chi=\{1,e\}$ and  $\LL_p^\chi$ consists of 2 isolated points. According to 
Theorem~\ref{teo:free}, the free rank of $\H_1(A_\Gamma^\chi,\FF)$ is
$$\sum_{A_X\in\mathcal{B}_p^\chi,|X|\leq 1}\dim_\FF
\overline{\H}_{0-|X|}(\hat\lk_{\LL^\chi_p}(X),\FF)=\dim_\FF\overline{\H}_{0}(\hat\lk_{\LL^\chi_p}(\emptyset),\FF)=
\dim_\FF\overline{\H}_{0}(\hat\LL^\chi_p,\FF)=1.$$
\end{exam}

\begin{exam} 
Let $G=\D_4\times \D_6$ where $\D_4$ (resp. $\D_6$) is the dihedral Artin group associated to the edge with label 4 (resp. 6). 
Then $G=A_\Gamma$, where $\Gamma$ is a full graph with 4 vertices and two disjoint edges labeled with 4 and 6.  
Denote the standard generators of the factor $\D_4$ by $v$, $w$ and the standard generators of the factor $\D_6$ 
by $x$, $y$ and consider the character $\chi:G\to\ZZ$ induced by $\chi(v)=\chi(x)=1$, $\chi(w)=\chi(y)=-1$. 
Taking into account the computation of the Sigma-invariants for this type of groups that we performed in 
Subsection~\ref{subsec:productdihedral}, we see that $[\chi]\not\in\Sigma^2(G,\ZZ)$ so its kernel $A_\Gamma^\chi$ 
is not of type $\FP_2$. In fact $G$ does not satisfy the strong $2$-link condition. To see it, note that 
$\mathcal{B}^\chi=\{\emptyset,e_1,e_2\}$ where $e_1=(v,w)$ and $e_2=(x,y)$ and $\LL^\chi$ is a square with 
vertices $v,x,w,y$ (that we get when we remove the interior of $e_1$ and $e_2$ from $\Gamma$). 
For $X=\emptyset\in\mathcal{B}^\chi$ we have $\lk_{\LL^\chi}(\emptyset)=\LL^\chi$. Then
$$\overline{\H}_{2-1-|X|}(\hat{\lk}_{\LL^\chi}(X))=\overline{\H}_{1}(\hat{\LL^\chi})=\ZZ\neq 0.$$
It is easy to see that also the strong $3$-link condition fails: to check it consider for example $X=e_1$, 
its link in $\LL^\chi$ consists of the isolated vertices $x$ and $y$.

We claim however that this $\chi$ does satisfy the strong $p$-$n$-link condition for each $p$ (zero or a prime). 
As a consequence, for any field $\FF$,
$$\dim_\FF\H_2(A_\Gamma^\chi,\FF)<\infty.$$

Assume first that $p=2$. Then $\mathcal{B}_2^\chi=\{\emptyset,e_1\}$ and $\LL_2^\chi$ is the graph obtained from $\Gamma$ 
when we remove the open edge $e_1$. Then $\lk_{\LL_2^\chi}(\emptyset)=\LL_2^\chi$ and $\lk_{\LL_2^\chi}(e_1)=e_2$ and both associated flag complexes are contractible. 

The argument for $p=3$ is analogous. Finally, if $p\neq 2,3$, $\mathcal{B}_p^\chi=\{\emptyset\}$ and $\LL_p^\chi=\Gamma$. Again, the flag complex is contractible. 

\end{exam}

\begin{exam} Things are very different if we consider for example $G_1=\D_4\times \D_4$ and $\chi$ as before. 
Then one easily checks that the strong $2$-$2$-link condition fails so $\dim_\FF\H_2(A_\Gamma^\chi,\FF)$ is infinite.
\end{exam}

\section{The homotopic invariants}
\label{sec:homotopic}
In this section we explain how to modify the statement of Theorem~\ref{teo:mainsigma} to obtain the analogous homotopic result.

Basically, we have to change the hypothesis to the homotopic version. As we have said in Definition~\ref{nlink}, 
we define the strong homotopic $n$-link condition as follows:

Consider again the set $\mathcal{B}^\chi\subset \cP$ of those $A_\Delta$  in the clique poset such that  $\chi(Z(A_\Delta))=0$. 
Assume that for any $A_\Delta\in\mathcal{B}^\chi$ with $|\Delta|\leq n$ the link $\lk_{\LL^\chi}(\sigma)$ 
is $(n-1-|\sigma|)$-connected. Then we say that $\chi$ satisfies \emph{the strong homotopic $n$-link condition}. 

\begin{teo}\label{teo:mainsigmahomotopic} 
Let  $G=A_\Gamma$ be an even Artin group of $\FC$-type, and $0\neq\chi:G\to\RR$ a character such that the strong homotopic
$n$-link condition holds for $\chi$. Then $[\chi]\in\Sigma^n(G)$.
\end{teo}

\begin{proof}
The proof follows that of Theorem~\ref{teo:mainsigma} in its homotopic version. 
The homotopic analogue of Lemma~\ref{lem:cosetposets} can be proved using relative homotopy groups, 
and the $(n-1)$-connectedness of $Z^S\star|\mathcal{J}^S|$ (see~\cite[p.57 (2.5)]{Whitehead-homotopy}).
\end{proof}

\section*{Acknowledgements}
The authors would like to thank the anonymous referee for a number of comments and suggestions that helped 
improve the exposition.

\bibliographystyle{amsplain}
\bibliography{referencias}

\end{document}